\documentclass[a4paper,12pt]{article}
\usepackage{graphicx}
\usepackage{amsmath}
\usepackage{amsthm}
\usepackage{amssymb}
\usepackage{cite}


\newtheorem{theorem}{Theorem}

\newtheorem{corollary}{Corollary}

\title{Spectral Properties of the Threshold Network Model}
\author{
Yusuke Ide
\footnote{
E-mail: ide@kanagawa-u.ac.jp
}
\\
\small{\textit{Faculty of Engineering, Kanagawa University,}}
\\
\small{\textit{Yokohama, 221-8686, Japan}}
\\
\\
Norio Konno
\footnote{
E-mail: konno@ynu.ac.jp
}
\\
\small{\textit{Department of Applied Mathematics, Yokohama National University,}}
\\
\small{\textit{Yokohama, 240-8501, Japan}}
\\
\\
Nobuaki Obata
\footnote{
E-mail: obata@math.is.tohoku.ac.jp
\newline 
\newline 
\quad \ 
{\it Abbr. title:} 
Spectral Properties of Threshold Network Model
\newline 
\quad \ 
{\it Key words and phrases:} 
spectral distribution, threshold network model, random graphs, complex networks.
}
\\
\small{\textit{Graduate School of Information Sciences, Tohoku University,}}
\\
\small{\textit{Sendai, 980-8579, Japan}}
}

\date{}
\begin{document}
\maketitle

\begin{abstract}
We study the spectral distribution of the threshold network model.
The results contain an explicit description and its asymptotic behaviour.
\end{abstract}

%%%%%%%%%%%%%%%%%%%%%%%%%%%%%%%%%%%%%%%%%%%%%%%%%%%%%%%%%%%%%%%%%%%%%%%%%%%%%%%%%%%%%%%
\section{Introduction}
%%%%%%%%%%%%%%%%%%%%%%%%%%%%%%%%%%%%%%%%%%%%%%%%%%%%%%%%%%%%%%%%%%%%%%%%%%%%%%%%%%%%%%% 

The \textit{threshold network model} $\mathcal{G}_{n}(X,\theta)$, 
where $X$ is a random variable, $n\ge2$ is an integer and $\theta\in \mathbb{R}$ is 
a constant called a threshold,
is a random graph on the vertex set $V=\{1,2,\dots ,n\}$ obtained as follows:
let $X_{1},X_{2},\dots,X_{n}$ be independent copies of $X$
and draw an edge between two distinct vertices $i,j\in V$ if $X_{i}+X_{j}>\theta$.
In other words, $\mathcal{G}_{n}(X,\theta)$ is specified 
by the random adjacency matrix $A=(A_{ij})$
defined by  
\begin{eqnarray*}\label{adjTM}
A_{ij}=
\begin{cases}
I_{(\theta ,\infty)}(X_{i}+X_{j}),& \text{if $i\neq j$},\\
0, & \text{otherwise},
\end{cases}
\end{eqnarray*}
where $I_{B}$ denotes the indicator function of a set $B$.
As a small variant one may allow self-loops, see e.g., \cite{BoseSen2007}.
In this case the threshold network model is denoted by $\Tilde{\mathcal{G}}_{n}(X,\theta)$, 
where two vertices $i,j\in V$ (possibly $i=j$) are connected if $X_{i}+X_{j}>\theta$.
The adjacency matrix $\Tilde{A}=(\Tilde{A}_{ij})$ is given by
\begin{equation*}\label{adjTM-1}
\Tilde{A}_{ij}=I_{(\theta ,\infty)}(X_{i}+X_{j}),
\qquad i,j\in V.
\end{equation*}

The threshold network model has been extensively studied
as a reasonable candidate model of real world complex graphs (networks),
which are often characterized by small diameters, high clustering, and 
power-law (scale-free) degree distributions \cite{Albert02,NewmanSIAM,blmch}. 
In fact, the threshold network model belongs to the so-called hidden variable models \cite{ccdm,so}
and is known for being capable of generating scale-free networks. 
Their mean behavior \cite{bps,ccdm,fum,hss,mmk04,scb,so} and limit theorems 
\cite{kmrs05,ikm07rims,ikm09MCAP,fikmmu09IIS} for the degree, the clustering coefficients, 
the number of subgraphs, and the average distance have been analyzed.
See also \cite{dhj09,ikm07rims,ikm09MCAP,kmrs05,mp95,mmk05,mk06} for related works.

Spectral properties of the threshold network model are also of interest.
As a simple case, the binary threshold model appears in \cite{Taraskin}.
The strong law of large numbers and central limit theorem for the rank of 
the adjacency matrix of the model with self-loops are given by \cite{BoseSen2007}. 
Eigenvalues and eigenvectors 
of the Laplacian matrix of the model have been studied \cite{Merris1994,Merris1998}. 
For general results of spectral analysis of graphs see e.g. Hora--Obata\cite{HO}. 
The main purpose of this paper is to study the spectral distribution of the threshold network model.
Our result covers the preceding study of the rank of the adjacency matrix.

This paper is organized as follows:
In Section 2 we recall the hierarchical structure of the threshold network model 
and derive the spectral distribution of each sample graph (threshold graph).
In Section 3 we obtain similar results for the threshold network model which admits self-loops. 
In Section 4 we derive some asymptotic behaviors of the spectral distributions
and in Section 5 we give a simple example called the binary threshold model.

%%%%%%%%%%%%%%%%%%%%%%%%%%%%%%%%%%%%%%%%%%%%%%%%%%%%%%%%%%%%%%%%%%%%%%%%%%%%%%%%%%%%%%%
\section{Spectra of threshold graphs}
\label{sec:Spectra of the threshold network model}
%%%%%%%%%%%%%%%%%%%%%%%%%%%%%%%%%%%%%%%%%%%%%%%%%%%%%%%%%%%%%%%%%%%%%%%%%%%%%%%%%%%%%%% 

Each sample graph $G\in\mathcal{G}_{n}(X,\theta)$ has a hierarchical structure 
described by the so-called creation sequence, introduced by Hagberg--Schult--Swart \cite{hss}.
Here we adopt a variant by Diaconis--Holmes--Janson \cite{dhj09}.
Each $G$ being determined by the values of random variables $X_1,X_2,\dots,X_n$,
we arrange them in increasing order:
$X_{(1)}\leq X_{(2)}\leq \cdots \leq X_{(n)}$.
If $X_{(1)}+X_{(n)}>\theta$, we have
\[
\theta <X_{(1)}+X_{(n)}\leq X_{(2)}+X_{(n)}\leq 
\dots \leq X_{(n-1)}+X_{(n)},
\]
which means that the vertex corresponding to $X_{(n)}$ is connected with 
the $n-1$ other vertices. 
Otherwise, we have 
\begin{eqnarray*}
\theta\ge X_{(1)}+X_{(n)}\ge \dots \ge X_{(1)}+X_{(3)}
\ge X_{(1)}+X_{(2)},
\end{eqnarray*}
which means that the vertex corresponding to $X_{(1)}$ is isolatedD
We set $s_n=1$ or $s_n=0$ according as the former case or the latter occurs.
Then, according to the case we remove the random variable $X_{(n)}$ or $X_{(1)}$,
we continue similar procedure to define $s_{n-1},\dots,s_2$.
Finally, we set $s_1=s_2$ and obtain a $\{0,1\}$-sequence $\{s_1,s_2,\dots,s_n\}$,
which is called the \textit{creation sequence} of $G$ and is denoted by $S_G$.

Given a creation sequence $S_G$ let $k_{i}$ and $l_{i}$ denote the number of consecutive bits of $1$'s 
and $0$'s, respectively, as follows: 
\begin{equation}\label{dfkl}
S_{G}=\{
\overbrace{1,\ldots ,1}^{k_{1}},\overbrace{0,\ldots ,0}^{l_{1}},
\overbrace{1,\ldots ,1}^{k_{2}},\overbrace{0,\ldots ,0}^{l_{2}},
\ldots ,
\overbrace{1,\ldots ,1}^{k_{m}},\overbrace{0,\ldots ,0}^{l_{m}}
\}.
\end{equation}
It may happen that $k_{1}=0$ or $l_{m}=0$,
but we have $k_2,\dots,k_m,l_1,\dots,l_{m-1}\ge1$ and $m\ge1$.
Moreover, by definition we have two cases:
(a) $k_1=0$ (equivalently $s_1=0$) and $l_1\ge2$;
(b) $k_1\ge2$ (equivalently, $s_1=1$).

For example, if $S_{G}=\{1,1,0,0,1,0,1,0\}$ then 
$k_{1}=2,\ l_{1}=2,\ k_{2}=1,\ l_{2}=1,\ k_{3}=1,\ l_{3}=1$ and 
Fig.\ \ref{kaisou3} shows the shape of $G$. 
\begin{figure}[htbp]
\begin{center}
%WinTpicVersion2.15
\unitlength 0.1in
\begin{picture}(16.00,14.00)(0.40,-14.00)
% DOT 0 0 3 0
% 9 400 805 400 1205 200 1605 600 1605 1200 805 1200 1205 1000 1605 1400 1605 1400 1605
% 
\special{pn 20}%
\special{sh 1}%
\special{ar 400 405 10 10 0  6.28318530717959E+0000}%
\special{sh 1}%
\special{ar 400 805 10 10 0  6.28318530717959E+0000}%
\special{sh 1}%
\special{ar 200 1205 10 10 0  6.28318530717959E+0000}%
\special{sh 1}%
\special{ar 600 1205 10 10 0  6.28318530717959E+0000}%
\special{sh 1}%
\special{ar 1200 405 10 10 0  6.28318530717959E+0000}%
\special{sh 1}%
\special{ar 1200 805 10 10 0  6.28318530717959E+0000}%
\special{sh 1}%
\special{ar 1000 1205 10 10 0  6.28318530717959E+0000}%
\special{sh 1}%
\special{ar 1400 1205 10 10 0  6.28318530717959E+0000}%
\special{sh 1}%
\special{ar 1400 1205 10 10 0  6.28318530717959E+0000}%
% STR 2 0 3 0
% 3 1200 400 1200 500 5 0
% $1$
\put(12.0000,-1.0000){\makebox(0,0){$1$}}%
% STR 2 0 3 0
% 3 400 410 400 510 5 0
% $0$
\put(4.0000,-1.1000){\makebox(0,0){$0$}}%
% LINE 2 2 3 0
% 16 40 600 1640 600 1640 1800 40 1800 40 1400 1640 1400 1640 1000 40 1000 40 400 40 1800 840 1800 840 400 1640 400 1640 1800 1640 400 40 400
% 
\special{pn 8}%
\special{pa 40 200}%
\special{pa 1640 200}%
\special{dt 0.045}%
\special{pa 1640 200}%
\special{pa 1639 200}%
\special{dt 0.045}%
\special{pa 1640 1400}%
\special{pa 40 1400}%
\special{dt 0.045}%
\special{pa 40 1400}%
\special{pa 41 1400}%
\special{dt 0.045}%
\special{pa 40 1000}%
\special{pa 1640 1000}%
\special{dt 0.045}%
\special{pa 1640 1000}%
\special{pa 1639 1000}%
\special{dt 0.045}%
\special{pa 1640 600}%
\special{pa 40 600}%
\special{dt 0.045}%
\special{pa 40 600}%
\special{pa 41 600}%
\special{dt 0.045}%
\special{pa 40 0}%
\special{pa 40 1400}%
\special{dt 0.045}%
\special{pa 40 1400}%
\special{pa 40 1399}%
\special{dt 0.045}%
\special{pa 840 1400}%
\special{pa 840 0}%
\special{dt 0.045}%
\special{pa 840 0}%
\special{pa 840 1}%
\special{dt 0.045}%
\special{pa 1640 0}%
\special{pa 1640 1400}%
\special{dt 0.045}%
\special{pa 1640 1400}%
\special{pa 1640 1399}%
\special{dt 0.045}%
\special{pa 1640 0}%
\special{pa 40 0}%
\special{dt 0.045}%
\special{pa 40 0}%
\special{pa 41 0}%
\special{dt 0.045}%
% LINE 2 0 3 0
% 22 1400 1610 1000 1610 1000 1610 1200 1210 1200 1210 1400 1610 1400 1610 1200 810 1200 810 1200 1210 1000 1610 1200 810 1200 810 400 1210 200 1610 1200 810 1200 810 600 1610 600 1610 1200 1210 1200 1210 200 1610
% 
\special{pn 8}%
\special{pa 1400 1210}%
\special{pa 1000 1210}%
\special{fp}%
\special{pa 1000 1210}%
\special{pa 1200 810}%
\special{fp}%
\special{pa 1200 810}%
\special{pa 1400 1210}%
\special{fp}%
\special{pa 1400 1210}%
\special{pa 1200 410}%
\special{fp}%
\special{pa 1200 410}%
\special{pa 1200 810}%
\special{fp}%
\special{pa 1000 1210}%
\special{pa 1200 410}%
\special{fp}%
\special{pa 1200 410}%
\special{pa 400 810}%
\special{fp}%
\special{pa 200 1210}%
\special{pa 1200 410}%
\special{fp}%
\special{pa 1200 410}%
\special{pa 600 1210}%
\special{fp}%
\special{pa 600 1210}%
\special{pa 1200 810}%
\special{fp}%
\special{pa 1200 810}%
\special{pa 200 1210}%
\special{fp}%
\end{picture}%
\caption{A threshold graph $G$ corresponding to $S_{G}=\{1,1,0,0,1,0,1,0\}$}
\label{kaisou3}
\end{center}
\end{figure}
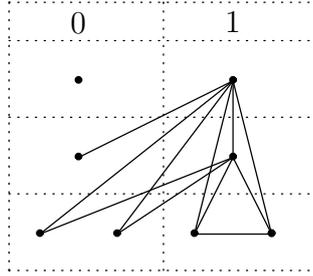

The creation sequence $S_G$ gives rise to a partition of the vertex set:
\[
V=\bigcup_{i=1}^m V^{(1)}_i \cup \bigcup_{i=1}^m V^{(0)}_i\,
\qquad |V^{(1)}_{i}|=k_i\,,
\quad |V^{(0)}_{i}|=l_i\,.
\]
The subgraph induced by $V^{(1)}_i$ is the complete graph of $k_i$ vertices,
and that induced by $V^{(0)}_i$ is the null graph of $l_{i}$ vertices.
Moreover, every vertex in $V^{(1)}_i$ (resp. $V^{(0)}_i$) is connected 
(resp. disconnected) with all vertices in 
\[
V^{(1)}_1\cup\dots\cup V^{(1)}_{i} \cup
V^{(0)}_1\cup\dots\cup V^{(0)}_{i-1}\,.
\]
In general, a graph possessing the above hierarchical structure is 
called a \textit{threshold graph} \cite{mp95}.

\begin{theorem}\label{thm1}
Let $G$ be a thereshold graph with a creation sequence $S_{G}=\{s_1=s_2,s_{3},\dots,s_n\}$.
Define $k_i$ and $l_i$ as in \eqref{dfkl} and set
\begin{equation}\label{defcoef}
C_{n}(-1)=\sum_{i=1}^{m}k_{i}-(m-1)-I_{\{1\}}(s_{1}),
\qquad
C_{n}(0)=\sum_{i=1}^{m}l_{i}-(m-1).
\end{equation}
Then the spectral distribution of $G$ is given by
\begin{equation}\label{eqn1:thm1}
\mu_{n}(G)
=\frac{C_{n}(-1)}{n}\, \delta_{-1}
 +\frac{C_{n}(0)}{n}\, \delta_{0}
 +\frac{1}{n} \sum_{j=1}^{J} \delta_{\lambda_j}\,,
 \quad J=2(m-1)+I_{\{1\}}(s_1),
\end{equation}
where $\{\lambda_j\}$ exhausts the eigenvalues of the matrix:
\begin{equation}\label{eqeigenvalue}
\begin{bmatrix}
k_{m}-1  & l_{m-1} & k_{m-1} & l_{m-2} & \hdots & l_{1} & k_{1}\\
k_{m} & 0 & 0 & 0 & \hdots & 0 & 0\\
k_{m} & 0 & k_{m-1}-1 & l_{m-2} & \hdots & l_{1} & k_{1}\\
k_{m} & 0 & k_{m-1} & 0 & \hdots & 0 & 0\\
\vdots & \vdots & \vdots & \vdots & \ddots & \vdots & \vdots\\
k_{m} & 0 & k_{m-1} & 0 & \hdots & 0  & 0\\
k_{m} & 0 & k_{m-1} & 0 & \hdots & 0 & k_{1}-1 
\end{bmatrix}
\end{equation}
for $s_{1}=1$ (equivalently, $k_1\ge2$), or 
\begin{equation}\label{eqeigenvalue-2}
\begin{bmatrix}
k_{m}-1  & l_{m-1} & k_{m-1} & l_{m-2} & \hdots & k_{2} & l_{1}\\
k_{m} & 0 & 0 & 0 & \hdots & 0 & 0\\
k_{m} & 0 & k_{m-1}-1 & l_{m-2} & \hdots & k_{2} & l_{1}\\
k_{m} & 0 & k_{m-1} & 0 & \hdots & 0 & 0\\
\vdots & \vdots & \vdots & \vdots & \ddots & \vdots & \vdots \\
k_{m} & 0 & k_{m-1} & 0 & \hdots & k_{2}-1 & l_{1}\\
k_{m} & 0 & k_{m-1} & 0 & \hdots & k_{2} & 0 
\end{bmatrix}
\end{equation}
for $s_{1}=0$ (equivalently, $k_1=0$).
Moreover, any $\lambda_j$ in \eqref{eqn1:thm1} differs from $0$ and $-1$.
\end{theorem}

\begin{proof}
Let $\boldsymbol{1}_{i,j}$ denote the $i\times j$ matrix consisting of only $1$, 
$\boldsymbol{0}_{i,j}$ the $i\times j$ zero matrix,
$I_{i}$ the $i\times i$ identity matrix,
and $\bar{\boldsymbol{1}}_{i,i}=\boldsymbol{1}_{i,i}-I_{i}$.
By the hierarchical structure mentioned above,
the adjacency matrix of $G$ is represented in the form:
\begin{eqnarray*}
A_{G}=
\begin{bmatrix}
\boldsymbol{0}_{l_{m},l_{m}} & \boldsymbol{0}_{l_{m},k_{m}} & \boldsymbol{0}_{l_{m},l_{m-1}} & 
\boldsymbol{0}_{l_{m},k_{m-1}} & \boldsymbol{0}_{l_{m},l_{m-2}} 
& \hdots & \boldsymbol{0}_{l_{m},l_{1}} & \boldsymbol{0}_{l_{m},k_{1}}\\
\boldsymbol{0}_{k_{m},l_{m}} & \bar{\boldsymbol{1}}_{k_{m},k_{m}} & \boldsymbol{1}_{k_{m},l_{m-1}} & 
\boldsymbol{1}_{k_{m},k_{m-1}} & \boldsymbol{1}_{l_{m},l_{m-2}} 
& \hdots & \boldsymbol{1}_{k_{m},l_{1}} & \boldsymbol{1}_{k_{m},k_{1}}\\
\boldsymbol{0}_{l_{m-1},l_{m}} & \boldsymbol{1}_{l_{m-1},k_{m}} & \boldsymbol{0}_{l_{m-1},l_{m-1}} & 
\boldsymbol{0}_{l_{m-1},k_{m-1}} & \boldsymbol{0}_{l_{m-1},l_{m-2}} 
& \hdots & \boldsymbol{0}_{l_{m-1},l_{1}} & \boldsymbol{0}_{l_{m-1},k_{1}}\\
\boldsymbol{0}_{k_{m-1},l_{m}} & \boldsymbol{1}_{k_{m-1},k_{m}} & \boldsymbol{0}_{k_{m-1},l_{m-1}} & 
\bar{\boldsymbol{1}}_{k_{m-1},k_{m-1}} & \boldsymbol{1}_{k_{m-1},l_{m-2}} 
& \hdots & \boldsymbol{1}_{k_{m-1},l_{1}} & \boldsymbol{1}_{k_{m-1},k_{1}}\\
\boldsymbol{0}_{l_{m-2},l_{m}} & \boldsymbol{1}_{l_{m-2},l_{m}} & \boldsymbol{0}_{l_{m-2},l_{m-1}} & 
\boldsymbol{1}_{l_{m-2},k_{m-1}} & \boldsymbol{0}_{l_{m-2},l_{m-2}} 
& \hdots & \boldsymbol{0}_{l_{m-2},l_{1}} & \boldsymbol{0}_{l_{m-2},k_{1}}\\
\vdots & \vdots & \vdots & \vdots & \vdots & \ddots & \vdots & \vdots\\
\boldsymbol{0}_{l_{1},l_{m}} & \boldsymbol{1}_{l_{1},k_{m}} & \boldsymbol{0}_{l_{1},l_{m-1}} & 
\boldsymbol{1}_{l_{1},k_{m-1}} & \boldsymbol{0}_{l_{1},l_{m-2}} 
& \hdots & \boldsymbol{0}_{l_{1},l_{1}} & \boldsymbol{0}_{l_{1},k_{1}}\\
\boldsymbol{0}_{k_{1},l_{m}} & \boldsymbol{1}_{k_{1},k_{m}} & \boldsymbol{0}_{k_{1},l_{m-1}} & 
\boldsymbol{1}_{k_{1},k_{m-1}} & \boldsymbol{0}_{k_{1},l_{m-2}} 
& \hdots & \boldsymbol{0}_{k_{1},l_{1}} & \bar{\boldsymbol{1}}_{k_{1},k_{1}}
\end{bmatrix}.
\end{eqnarray*}
The adjacency matrix $A$ acts on $\mathbb{C}^n$ from the left.
We define subspaces of $\mathbb{C}^n$ by 
\begin{align*}
V_{i}(-1)
&=
\left\{
\begin{bmatrix}
\boldsymbol{0}_{u_{i}+l_{i}}\\
\boldsymbol{\xi }_{k_{i}}\\
\boldsymbol{0}_{d_{i}}
\end{bmatrix}
\ :\ \xi _{1}+\xi _{2}+ \cdots +\xi _{k_{i}}=0 \right\}, \quad 1\leq i\leq m,\\
V_{i}(0)
&=
\left\{
\begin{bmatrix}
\boldsymbol{0}_{u_{i}}\\
\boldsymbol{\eta }_{l_{i}}\\
\boldsymbol{0}_{k_{i}+d_{i}}
\end{bmatrix}
\ :\ \eta _{1}+\eta _{2}+ \cdots +\eta _{l_{i}}=0 \right\},\quad 1\leq i\leq m-1,\\
V_{m}(0)
&=\left\{
\begin{bmatrix}
\boldsymbol{\eta }_{l_{m}}\\
\boldsymbol{0}_{k_{m}+d_{m}}
\end{bmatrix}
\right\},
\end{align*}
where
\begin{eqnarray*}
\boldsymbol{\xi }_{k}=
\begin{bmatrix}
\xi _{1}\\
\xi _{2}\\
\vdots \\
\xi _{k}
\end{bmatrix},
\quad 
\boldsymbol{\eta }_{l}=
\begin{bmatrix}
\eta _{1}\\
\eta _{2}\\
\vdots \\
\eta _{l}
\end{bmatrix},
\quad 
\boldsymbol{1}_{j}=
\begin{bmatrix}
1\\
1\\
\vdots \\
1
\end{bmatrix},
\quad 
\boldsymbol{0}_{j}=
\begin{bmatrix}
0\\
0\\
\vdots \\
0
\end{bmatrix}
\end{eqnarray*}
and
\[
u_{i}=\sum _{j=i+1}^{m}(l_{j}+k_{j}),
\qquad
d_{i}=\sum _{j=1}^{i-1}(l_{j}+k_{j}).
\]
Since $A_{G}$ acts on $V_{i}(-1)$ as the scalar operator with $-1$,
it possesses the eigenvalues $-1$ with multiplicity at least 
\[
\sum_{i=1}^m \dim V_i(-1)
=\sum_{i=1}^m (k_i-1)
=\sum_{i=1}^m k_i- m
\]
if $k_1\ge2$ (i.e., $s_1=1$), and
\[
\sum_{i=2}^m \dim V_i(-1)
=\sum_{i=2}^m (k_i-1)
=\sum_{i=2}^m k_i- (m-1)
\]
if $k_1=0$  (i.e., $s_1=0$).
In any case, the multiplicity is at least $C_n(-1)$ defined in \eqref{defcoef}.
Similarly, acting on $V_{i}(0)$ as a scalar operator with $0$,
$A_G$ possesses the eigenvalues $0$ with multiplicity at least $C_n(0)$.

Let $W$ be the orthogonal complement to 
$\bigoplus _{i=1}^{m}\left(V_{i}(-1)\oplus V_{i}(0)\right)$.
The matrix representation of $A_G$ on $W$ with respect to the basis  
\begin{eqnarray*}
\boldsymbol{v}_{i}=
\begin{bmatrix}
\boldsymbol{0}_{u_{i}+l_{i}}\\
\boldsymbol{1}_{k_{i}}\\
\boldsymbol{0}_{d_{i}}
\end{bmatrix}, \quad 1\leq i\leq m, 
\quad \text{and}\quad 
\boldsymbol{w}_{i}=
\begin{bmatrix}
\boldsymbol{0}_{u_{i}}\\
\boldsymbol{1}_{l_{i}}\\
\boldsymbol{0}_{k_{i}+d_{i}}
\end{bmatrix}, \quad 1\leq i\leq m-1. 
\end{eqnarray*} 
is given by \eqref{eqeigenvalue} or by \eqref{eqeigenvalue-2}
according as $k_1\ge2$ or $k_1=0$.
Then, one may verify easily the eigenvalues of 
the matrices \eqref{eqeigenvalue} and \eqref{eqeigenvalue-2} are different from 
$-1$ nor $0$.
\end{proof}

\noindent{\bfseries Remark}\enspace
After simple calculation we see that
the eigenvalues $\lambda_1,\dots,\lambda_J$ in \eqref{eqn1:thm1}
are obtained from the characteristic equations
\[
M(\lambda)=0,
\]
where
\[
M(\lambda)
=\det 
\begin{bmatrix}
k_{m}-1-\lambda  & l_{m-1} & k_{m-1} & l_{m-2} & \hdots & l_{1} & k_{1}\\
k_{m} & -\lambda & 0 & 0 & \hdots & 0 & 0\\
k_{m} & 0 & k_{m-1}-1-\lambda & l_{m-2} & \hdots & l_{1} & k_{1}\\
k_{m} & 0 & k_{m-1} & -\lambda & \hdots & 0 & 0\\
\vdots & \vdots & \vdots & \vdots & \ddots & \vdots & \vdots\\
k_{m} & 0 & k_{m-1} & 0 & \hdots & -\lambda  & 0\\
k_{m} & 0 & k_{m-1} & 0 & \hdots & 0 & k_{1}-1-\lambda 
\end{bmatrix}
\]
if $k_{1}\geq 2$ (i.e., $s_{1}=1$), and
\[
M(\lambda)
=\det 
\begin{bmatrix}
k_{m}-1-\lambda  & l_{m-1} & k_{m-1} & l_{m-2} & \hdots & k_{2} & l_{1}\\
k_{m} & -\lambda & 0 & 0 & \hdots & 0 & 0\\
k_{m} & 0 & k_{m-1}-1-\lambda & l_{m-2} & \hdots & k_{2} & l_{1}\\
k_{m} & 0 & k_{m-1} & -\lambda & \hdots & 0 & 0\\
\vdots & \vdots & \vdots & \vdots & \ddots & \vdots & \vdots \\
k_{m} & 0 & k_{m-1} & 0 & \hdots & k_{2}-1-\lambda & l_{1}\\
k_{m} & 0 & k_{m-1} & 0 & \hdots & k_{2} & -\lambda 
\end{bmatrix}
\]
if $k_{1}=0$ (i.e., $s_{1}=0$).
Simple calculation shows that
\[
M(-1)=
\begin{cases}
k_{1}\cdots k_{m}\cdot l_{1}\cdots l_{m-1}, &\text{if $k_{1}\geq 2$ (i.e., $s_{1}=1$)},\\
k_{2}\cdots k_{m}\cdot (l_{1}-1)\cdot l_{2}\cdots l_{m-1}, &\text{otherwise},
\end{cases}
\]
and
\[
M(0)=
\begin{cases}
(k_{1}-1)\cdot k_{2}\cdots k_{m}\cdot l_{1}\cdots l_{m-1},
&\text{if $k_{1}\geq 2$ (i.e., $s_{1}=1$)},\\
k_{2}\cdots k_{m}\cdot l_{1}\cdots l_{m-1}, &\text{otherwise},
\end{cases}
\]
from which we see also that $\{\lambda_j\}$ do not contain $-1$ or $0$.

%%%%%%%%%%%%%%%%%%%%%%%%%%%%%%%%%%%%%%%%%%%%%%%%%%%%%%%%%%%%%%%%%%%%%%%%%%%%%%%%%%%%%%%
\section{Spectra of threshold graphs with self-loops}
%%%%%%%%%%%%%%%%%%%%%%%%%%%%%%%%%%%%%%%%%%%%%%%%%%%%%%%%%%%%%%%%%%%%%%%%%%%%%%%%%%%%%%%

The idea of a creation sequence in Section \ref{sec:Spectra of the threshold network model}
can be applied to the threshold network model which allows self-loops.
With each $G\in \widetilde{\mathcal{G}}_{n}(X,\theta)$ we associate 
a creation sequence $\widetilde{S}_{G}=\{\tilde{s}_1,\tilde{s}_2, \dots, \tilde{s}_n\}$
as follows: if $X_{(1)}+X_{(n)}>\theta$, we have
\begin{eqnarray*}
\theta <X_{(1)}+X_{(n)}\leq X_{(2)}+X_{(n)}\leq \dots \leq X_{(n-1)}+X_{(n)}\leq X_{(n)}+X_{(n)},
\end{eqnarray*}
which implies that the vertex corresponding to $X_{(n)}$ is connected with the $n-1$ other vertices
and has a self-loop. 
Otherwise, 
\begin{eqnarray*}
\theta \ge X_{(1)}+X_{(n)}\ge \dots \ge
X_{(1)}+X_{(3)} \ge X_{(1)}+X_{(2)} \ge X_{(1)}+X_{(1)}\,,
\end{eqnarray*}
which means that the vertex corresponding to $X_{(1)}$ is isolated and has no self-loopsD
We set $\tilde{s}_n=1$ or $\tilde{s}_n=0$ according as the former case or the latter occurs.
Then, according to the case we remove the random variable $X_{(n)}$ or $X_{(1)}$,
we continue similar procedure to define $\tilde{s}_{n-1},\dots,\tilde{s}_2$. 
Finally, letting $X_{(*)}$ be the last remained random variable,
set $\tilde{s}_1=1$ if $X_{*}>\theta /2$ and $\tilde{s}_1=0$ otherwise. 
In this case $G$ is called a \textit{threshold graph with self-loops}
associated with a creation sequence $\Tilde{S}=\{\tilde{s}_1,\tilde{s}_2,\dots,\tilde{s}_n\}$.
We note that if $\tilde{s}_j=1$ the corresponding vertex has a self-loop
and otherwise no self-loop.

Given a creation sequence $\Tilde{S}=\{\tilde{s}_1,\tilde{s}_2,\dots,\tilde{s}_n\}$ 
we define $k_j$ and $l_j$ as in \eqref{dfkl}.
It may happen that $k_1=0$ and $l_m=0$,
but $k_2, \dots, k_m, l_1,\dots,l_{m-1}\ge1$ and $m\ge1$.
The adjacency matrix of $G$ is of the form:
\begin{equation}\label{eqn:adjacenct matrix of Tilde G}
\tilde{A}_G=
\begin{bmatrix}
\boldsymbol{0}_{l_{m},l_{m}} & \boldsymbol{0}_{l_{m},k_{m}} & \boldsymbol{0}_{l_{m},l_{m-1}} & 
\boldsymbol{0}_{l_{m},k_{m-1}} & \boldsymbol{0}_{l_{m},l_{m-2}} 
& \hdots & \boldsymbol{0}_{l_{m},l_{1}} & \boldsymbol{0}_{l_{m},k_{1}}\\
\boldsymbol{0}_{k_{m},l_{m}} & \boldsymbol{1}_{k_{m},k_{m}} & \boldsymbol{1}_{k_{m},l_{m-1}} & 
\boldsymbol{1}_{k_{m},k_{m-1}} & \boldsymbol{1}_{l_{m},l_{m-2}} 
& \hdots & \boldsymbol{1}_{k_{m},l_{1}} & \boldsymbol{1}_{k_{m},k_{1}}\\
\boldsymbol{0}_{l_{m-1},l_{m}} & \boldsymbol{1}_{l_{m-1},k_{m}} & \boldsymbol{0}_{l_{m-1},l_{m-1}} & 
\boldsymbol{0}_{l_{m-1},k_{m-1}} & \boldsymbol{0}_{l_{m-1},l_{m-2}} 
& \hdots & \boldsymbol{0}_{l_{m-1},l_{1}} & \boldsymbol{0}_{l_{m-1},k_{1}}\\
\boldsymbol{0}_{k_{m-1},l_{m}} & \boldsymbol{1}_{k_{m-1},k_{m}} & \boldsymbol{0}_{k_{m-1},l_{m-1}} & 
\boldsymbol{1}_{k_{m-1},k_{m-1}} & \boldsymbol{1}_{k_{m-1},l_{m-2}} 
& \hdots & \boldsymbol{1}_{k_{m-1},l_{1}} & \boldsymbol{1}_{k_{m-1},k_{1}}\\
\boldsymbol{0}_{l_{m-2},l_{m}} & \boldsymbol{1}_{l_{m-2},l_{m}} & \boldsymbol{0}_{l_{m-2},l_{m-1}} & 
\boldsymbol{1}_{l_{m-2},k_{m-1}} & \boldsymbol{0}_{l_{m-2},l_{m-2}} 
& \hdots & \boldsymbol{0}_{l_{m-2},l_{1}} & \boldsymbol{0}_{l_{m-2},k_{1}}\\
\vdots & \vdots & \vdots & \vdots & \vdots & \ddots & \vdots & \vdots\\
\boldsymbol{0}_{l_{1},l_{m}} & \boldsymbol{1}_{l_{1},k_{m}} & \boldsymbol{0}_{l_{1},l_{m-1}} & 
\boldsymbol{1}_{l_{1},k_{m-1}} & \boldsymbol{0}_{l_{1},l_{m-2}} 
& \hdots & \boldsymbol{0}_{l_{1},l_{1}} & \boldsymbol{0}_{l_{1},k_{1}}\\
\boldsymbol{0}_{k_{1},l_{m}} & \boldsymbol{1}_{k_{1},k_{m}} & \boldsymbol{0}_{k_{1},l_{m-1}} & 
\boldsymbol{1}_{k_{1},k_{m-1}} & \boldsymbol{0}_{k_{1},l_{m-2}} 
& \hdots & \boldsymbol{0}_{k_{1},l_{1}} & \boldsymbol{1}_{k_{1},k_{1}}
\end{bmatrix}.
\end{equation}
Repeating a similar argument as in Theorem \ref{thm1}, we come to the following

\begin{theorem}\label{thm2}
Let $G$ be a threshold graph with self-loops
associated with a creation sequence $\Tilde{S}=\{\tilde{s}_1,\tilde{s}_2,\dots,\tilde{s}_n\}$
and its adjacency matrix given as in \eqref{eqn:adjacenct matrix of Tilde G}.
Set
\begin{eqnarray}\label{defcoef2}
\widetilde{C}_{n}(0)=n-2(m-1)-I_{\{1\}}(\tilde{s}_{1}).
\end{eqnarray}
Then the spectral distribution of $G$ is given by
\begin{equation}\label{eqn1:thm2}
\widetilde{\mu }_{n}(G)
=\frac{\widetilde{C}_{n}(0)}{n}\, \delta_{0}
 +\frac{1}{n} \sum_{j=1}^J \delta_{\lambda_j}\,,
 \quad J=2(m-1)+I_{\{1\}}(\tilde{s}_{1})
\end{equation}
where $\{\lambda_j\}$ exhaust the eigenvalues of
\begin{equation*}\label{eqeigenvalue2}
\begin{bmatrix}
k_{m}  & l_{m-1} & k_{m-1} & l_{m-2} & \hdots & l_{1} & k_{1}\\
k_{m} & 0 & 0 & 0 & \hdots & 0 & 0\\
k_{m} & 0 & k_{m-1} & l_{m-2} & \hdots & l_{1} & k_{1}\\
k_{m} & 0 & k_{m-1} & 0 & \hdots & 0 & 0\\
\vdots & \vdots & \vdots & \vdots & \ddots & \vdots & \vdots\\
k_{m} & 0 & k_{m-1} & 0 & \hdots & 0  & 0\\
k_{m} & 0 & k_{m-1} & 0 & \hdots & 0 & k_1 
\end{bmatrix}
\end{equation*}
for $\tilde{s}_{1}=1$ (i.e., $k_1\ge1$), or
\begin{equation*}\label{eqeigenvalue2-2}
\begin{bmatrix}
k_{m}  & l_{m-1} & k_{m-1} & l_{m-2} & \hdots & k_{2} & l_{1}\\
k_{m} & 0 & 0 & 0 & \hdots & 0 & 0\\
k_{m} & 0 & k_{m-1} & l_{m-2} & \hdots & k_{2} & l_{1}\\
k_{m} & 0 & k_{m-1} & 0 & \hdots & 0 & 0\\
\vdots & \vdots & \vdots & \vdots & \ddots & \vdots & \vdots \\
k_{m} & 0 & k_{m-1} & 0 & \hdots & k_{2} & l_{1}\\
k_{m} & 0 & k_{m-1} & 0 & \hdots & k_{2} & 0 
\end{bmatrix}
\end{equation*}
for $\tilde{s}_{1}=0$ (i.e., $k_1=0$).
Moreover, any $\lambda_j$ in \eqref{eqn1:thm2} differs from $0$.
\end{theorem}

\noindent{\bfseries Remark}\enspace
The eigenvalues $\lambda_1,\dots,\lambda_J$ in \eqref{eqn1:thm2} 
are obtained from the characteristic equations:
\begin{eqnarray*}\label{eqn2-1:thm2}
\det 
\begin{bmatrix}
-\lambda & 0 & \hdots & 0 & 0 & 0 & \hdots & 0 & k_{m}\\
\lambda & -\lambda & \hdots & 0 & 0 & 0 & \hdots & k_{m-1} & 0\\
\vdots & \ddots & \ddots & \vdots & \vdots & \vdots & \hdots & \vdots & \vdots \\
0 & 0 & \ddots & -\lambda & 0 & k_{2} & \hdots & 0 & 0\\
0 & 0 & \hdots & \lambda & k_{1}-\lambda & 0 & \hdots & 0 & 0\\
0 & 0 & \hdots & l_{1} & \lambda & -\lambda & \hdots & 0 & 0\\
\vdots & \vdots & \vdots & \vdots & \vdots & \ddots & \ddots & \vdots & \vdots \\
0 & l_{m-2} & \hdots & 0 & 0 & 0 & \ddots & -\lambda & 0\\
l_{m-1} & 0 & \hdots & 0 & 0 & 0 & \hdots & \lambda & -\lambda 
\end{bmatrix}
=0
\end{eqnarray*}
for $s_{1}=1$ (i.e., $k_1\ge1$), or
\begin{eqnarray*}\label{eqn2-2:thm2}
\det 
\begin{bmatrix}
-\lambda & 0 & \hdots & 0 & 0 & 0 & 0 & \hdots & 0 & k_{m}\\
\lambda & -\lambda & \hdots & 0 & 0 & 0 & 0 & \hdots & k_{m-1} & 0\\
0 & \lambda & \ddots & 0 & 0 & 0 & 0 & \hdots & 0 & 0\\
\vdots & \vdots & \ddots & \ddots & \vdots & \vdots & \vdots & \hdots & \vdots & \vdots \\
0 & 0 & \hdots & \lambda & -\lambda & k_{2} & 0 & \hdots & 0 & 0\\
0 & 0 & \hdots & 0 & l_{1}+\lambda & -\lambda & 0 & \hdots & 0 & 0\\
0 & 0 & \hdots & l_{2} & 0 & \lambda & -\lambda & \hdots & 0 & 0\\
\vdots & \vdots & \vdots & \vdots & \vdots & \vdots & \ddots & \ddots & \vdots & \vdots \\
0 & l_{m-2} & \hdots & 0 & 0 & 0 & 0 & \ddots & -\lambda & 0\\
l_{m-1} & 0 & \hdots & 0 & 0 & 0 & 0 & \hdots & \lambda & -\lambda 
\end{bmatrix}
=0
\end{eqnarray*}
for $s_{1}=0$ (i.e., $k_1=0$).

%%%%%%%%%%%%%%%%%%%%%%%%%%%%%%%%%%%%%%%%%%%%%%%%%%%%%%%%%%%%%%%%%%%%%%%%%%%%%%%%%%%%%%%
\section{Limit theorems}
%%%%%%%%%%%%%%%%%%%%%%%%%%%%%%%%%%%%%%%%%%%%%%%%%%%%%%%%%%%%%%%%%%%%%%%%%%%%%%%%%%%%%%% 

In this section we discuss asymptotic behaviors of the spectral distributions
obtained in the previous sections.

We first consider the case where the distribution of $X$ is discrete and given by
\[
\mathbb{P}(X=i)=p_i\,, \quad i=0,1,\dots,
\qquad \sum _{i=0}^{\infty }p_{i}=1.
\]
Let $m\ge1$ be a fixed integer.
Take a particlar threshold $\theta =2m-1$ and assume that $p_{i}>0$ for $i=0,1,\dots,2m-1$. 
It follows from the strong law of large numbers that 
\begin{align*}
l_{i} &=\sharp\{j:X_{j}=m-i\}, \quad i=1,\dots,m, \\
k_{i} &=\sharp\{j:X_{j}=m-1+i\}, \quad i=1,\dots,m-1, \\
k_{m} &=\sharp\{j:X_{j}\geq 2m-1\}, \\
l_{i} &=k_{i}=0, \quad i\geq m+1,
\end{align*}
for large $n$ almost surely. 
Moreover, denoting by $F$ the distribution function of $X$, we have  
\begin{align*}
&\lim_{n\to \infty }\frac{1}{n} \sum_{i=1}^{m}l_{i}=F(m-1) \quad \text{a.s.}
&\lim_{n\to \infty }\frac{1}{n} \sum_{i=1}^{m}k_{i}=1-F(m-1) \quad \text{a.s.}
\end{align*}
With these observation we easily obtain the following

\begin{theorem}\label{thmlimd}
Notations and assumptions being as above, 
the spectral distributions of $\mathcal{G}_{n}(X,2m-1)$ verifies
\[
\lim_{n\to \infty }\mu _{n}(G)=
\left(1-F(m-1)\right)\cdot \delta _{-1}+F(m-1)\cdot \delta _{0} \quad \text{a.s.}
\]
Similarly, the spectral distributions of $\tilde{\mathcal{G}}_{n}(X,2m-1)$ verifies
\[
\lim_{n\to \infty }\widetilde{\mu }_{n}(G)=\delta _{0} \quad \text{a.s.}
\]
\end{theorem}

\noindent{\bfseries Remark}\enspace
Similar results hold when the distribution of $X$ is discrete and 
$F$ has only finite number of jumps in $(-\infty ,\theta /2]$ or $(\theta /2,\infty)$. 
But no simple description is known for a general case.

\bigskip

Next we consider the case where the distribution of $X$ is continuous.
As is stated by Bose--Sen \cite{BoseSen2007} implicitly,
if the distribution of $X$ is continuous and symmetric around $0$,
then the distribution of zero and one entries in 
the creation sequence $\widetilde{S}$ of each graph generated by $\widetilde{\mathcal{G}}_{n}(X,0)$ 
is the same as the distribution of a sequence of i.i.d.\ Bernoulli random variables 
$\{\widetilde{Y}_{i}\}_{i=1,2,\ldots ,n}$ with success probability $1/2$, that is, 
\begin{eqnarray*}
\mathbb{P}(\widetilde{Y}_{i}=0)=\mathbb{P}(\widetilde{Y}_{i}=1)=1/2,
\quad \text{for $i=1,2,\ldots ,n$}.
\end{eqnarray*}
This means that 
$
\widetilde{S}=\{\tilde{s}_{1},\tilde{s}_{2},\dots ,\tilde{s}_{n}\}
\overset{d}{=}
\{\widetilde{Y}_{1},\widetilde{Y}_{2},\dots ,\widetilde{Y}_{n}\}.
$
Recall that the creation sequence $S$ 
of each graph generated by $\mathcal{G}_{n}(X,0)$ is always satisfied with $s_{1}=s_{2}$. 
Then we observe that 
$
S=\{s_{1}=s_{2},s_{3},\dots ,s_{n}\}
\overset{d}{=}
\{Y_{2},Y_{2},Y_{3},\dots ,Y_{n}\},
$
where $\{Y_{i}\}_{i=2,3,\ldots ,n}$ be the sequence of i.i.d.\ Bernoulli random variables 
with success probability $1/2$, similarly. 

Taking the above consideration into account, 
we obtain asymptotic behaviors of coefficients of point measures on $-1$ and $0$ 
appearing in $\mu _{n}(G)$ and $\widetilde{\mu }_{n}(G)$. 

\begin{theorem}
Assume that the distribution of $X$ is continuous and symmetric around $0$. 
Define $C_{n}(-1)$, $C_{n}(0)$ and $\widetilde{C}_{n}(0)$ as in \eqref{defcoef} and \eqref{defcoef2}.
Then we have 
\begin{enumerate}
\item[\upshape (1)] 
$\displaystyle\lim_{n\to \infty}
C_{n}(-1)/n
=\lim_{n\to \infty}
C_{n}(0)/n
=1/4 \quad \text{a.s.}$
\item[\upshape (2)] 
$\sqrt{n}\left(C_{n}(-1)/n-1/4\right)\Rightarrow N(0,1/4)$ and 
$\sqrt{n}\left(C_{n}(0)/n-1/4\right)\Rightarrow N(0,1/4)$ as $n\to \infty$.
\item[\upshape (3)] 
$\displaystyle\lim_{n\to \infty}
\widetilde{C}_{n}(0)/n=1/2 \quad \text{a.s.}$
\item[\upshape (4)] 
$\sqrt{n}(\widetilde{C}_{n}(0)/n-1/2) \Rightarrow N(0,1/4)$ as $n\to \infty$.
\end{enumerate}
\end{theorem}

\begin{proof}
Note the following relations:
\begin{align*}
C_{n}(-1)
&=
\sum_{i=1}^{m}k_{i}-(m-1)-I_{\{1\}}(s_{1})\\
&\overset{d}{=}
\left(Y_{2}+\sum_{i=2}^{n}Y_{i}\right)
-\sum_{i=2}^{n-1}(1-Y_{i})Y_{i+1}
-Y_{2}
=
Y_{2}+\sum_{i=2}^{n-1}Y_{i}Y_{i+1},\\
C_{n}(0)
&=
\sum_{i=1}^{m}l_{i}-(m-1)\\
&\overset{d}{=}
\left\{(1-Y_{2})+\sum_{i=2}^{n}(1-Y_{i})\right\}
-\sum_{i=2}^{n-1}(1-Y_{i})Y_{i+1}\\
&=
2-Y_{2}-Y_{n}+\sum_{i=2}^{n-1}(1-Y_{i})(1-Y_{i+1}),\\
\widetilde{C}_{n}(0)
&=
n-2(m-1)-I_{\{1\}}(\tilde{s}_{1})\\
&\overset{d}{=}
n-2\sum_{i=1}^{n-1}(1-\widetilde{Y}_{i})\widetilde{Y}_{i+1}-\widetilde{Y}_{1}
=
1-\widetilde{Y}_{n}
+\sum_{i=1}^{n-1}(1-\widetilde{Y}_{i}+\widetilde{Y}_{i+1})(1+\widetilde{Y}_{i}-\widetilde{Y}_{i+1}).
\end{align*}
We then easily check that
\[
\mathbb{E} [C_{n}(-1)]
=\mathbb{E}[C_{n}(0)]-\frac12
=\frac{n}{4}
\]
and 
\[
\mathbb{E}[\widetilde{C}_{n}(0)]=\frac{n}{2}\,.
\]
Applying a similar argument as in \cite[Theorem 1]{BoseSen2007}, we have the assertion.
\end{proof}

When the distribution of $X$ is continuous and symmetric around $\theta /2$,
we can obtain similar results for $\mathcal{G}_{n}(X,\theta )$ and 
$\widetilde{\mathcal{G}}_{n}(X,\theta )$ by straightforward modification.
Study covering a more general situation is now in progress.

%%%%%%%%%%%%%%%%%%%%%%%%%%%
\section{Binary threshold model}
%%%%%%%%%%%%%%%%%%%%%%%%%%%

In this section we give a simple example. 
The threshold network model defined by
Bernoulli trials $X_1,X_2,\dots,X_n$ with success probability $p$,
i.e., $0<P(X_i=1)=p<1$, and a threshold $0\le \theta<1$ is called the 
\textit{binary threshold model} and is denoted by $\mathcal{G}_n(p)$.
For $G\in \mathcal{G}_n(p)$ the partition of the vertex set $V$ is given by
\[
V=V^{(1)}\cup V^{(0)},
\qquad
V^{(1)}=\{i\,;\, X_i=1\},
\quad
V^{(0)}=\{i\,;\, X_i=0\}.
\]

\begin{theorem}\label{thm:spectrum of S(k,l)}
For $G\in \mathcal{G}_n(p)$ we set $|V^{(1)}|=k$ and $|V^{(0)}|=l$.
Then the spectral distribution of $G$ is given by 
\begin{equation*}\label{eqn:spec(S(k,l))}
\mu_{k,l}
=\frac{k-1}{n}\,\delta_{-1}
 +\frac{l-1}{n}\,\delta_0
 +\frac{1}{n}\,\delta_{\lambda_+}
 +\frac{1}{n}\,\delta_{\lambda_-}\,,
\end{equation*}
where 
\begin{equation}\label{eqn:lambda+-}
\lambda_{\pm}=\frac{k-1\pm \sqrt{(k-1)^2+4kl}}{2}\,.
\end{equation}
\end{theorem}

\begin{proof}
We need only to apply  Theorem \ref{thm1} with $l_{1}=l$, $l_{2}=k_{1}=0$, $k_{2}=k$
and $m=2$.
In this case, \eqref{eqeigenvalue-2} becomes
\[
\begin{bmatrix}
k-1 & l \\
k & 0 
\end{bmatrix},
\]
of which the eigenvalues are $\lambda_\pm$ in \eqref{eqn:lambda+-}.
\end{proof}

\begin{corollary}
Let $\mu _{n}(G)$ be the spectral distribution of $G\in \mathcal{G}_n(p)$. 
Then we have 
\begin{eqnarray*}
\lim _{n\to \infty }\mu _{n}(G)=p\cdot \delta _{-1}+(1-p)\cdot \delta _{0}\quad \text{a.s.}
\end{eqnarray*}
\end{corollary}

\begin{proof}
By the strong law of large numbers, see also Theorem \ref{thmlimd}.
\end{proof}

As for the the mean spectral distribution we have

\begin{theorem}
The mean spectral distribution of the binary threshold model $\mathcal{G}_{n}(p)$ 
is given by 
\begin{align}
\mu _{n}
&=\left(p-\frac{1}{n}\right)\delta_{-1}+\left(1-p-\frac{1}{n}\right)\delta_0 
\nonumber \\
&\qquad\quad +\frac{1}{n}\sum_{k=0}^n \binom{n}{k}p^k(1-p)^{n-k}
 \left(\delta_{\lambda_-(k)}+\delta_{\lambda_+(k)}\right),
\label{spec of BTM}
\end{align}
where
\[
\lambda_\pm(k)=\frac{k-1\pm\sqrt{(k-1)^2+4k(n-k)}}{2}\,,
\qquad k=0,1,\dots,n.
\]
\end{theorem}

\begin{proof}
Since
\[
P(|V^{(1)}|=k, |V^{(0)}|=l)=\binom{n}{k}p^k(1-p)^l,
\qquad k+l=n,
\]
the mean spectral distribution is given by
\[
\mu=\sum_{k=0}^{n} \binom{n}{k}p^k(1-p)^l \mu_{k,l}\,.
\]
Then \eqref{spec of BTM} follows 
from Theorem \ref{thm:spectrum of S(k,l)} by direct computation.
\end{proof}

\begin{corollary}
Let $\mu _{n}$ be mean spectral distribution of the binary threshold model $\mathcal{G}_{n}(p)$. 
Then we have 
\begin{eqnarray*}
\lim _{n\to \infty }\mu _{n}=p\cdot \delta _{-1}+(1-p)\cdot \delta _{0}\,.
\end{eqnarray*}
\end{corollary}

%%%%%%%%%%%%%%%%%%%%%%%%%%%%%%%%%%%%%%%%%%%%%%%%%%%%%%%%%%%%%%%%%%%%%%%%%%%%%%%%%%%%%%%%%%%
{\bf Acknowledgment.} 
NK is funded by 
the Grant-in-Aid for Scientific Research (C) of 
Japan Society for the Promotion of Science (Grant No. 21540118). 
NO is funded by 
the Grant-in-Aid for Challenging Exploratory Research of 
Japan Society for the Promotion of Science (Grant No. 19654024).
%%%%%%%%%%%%%%%%%%%%%%%%%%%%%%%%%%%%%%%%%%%%%%%%%%%%%%%%%%%%%%%%%%%%%%%%%%%%%%%%%%%%%%%%%%%

%%%%%%%%%%%%%%%%%%%%%%%%%%%%%%%%%%%%%%%%%%%%%%%%%%%%%%%%%%%%%%%%%%%%%%%%%%%%%%%%%%%%%%%%%%%%%
\end{document}